\documentclass[11pt,oneside,leqno]{amsart}

\makeatletter
\@namedef{subjclassname@2020}{\textup{2020} Mathematics Subject Classification}
\makeatother

\usepackage[top=2.5 cm,bottom=2 cm,left=2 cm,right=3 cm]{geometry}
\usepackage[utf8]{inputenc}
\usepackage{color}
\usepackage{amssymb}

\usepackage[english]{babel}

\usepackage{todonotes}
\usepackage[backref=page]{hyperref}

\usepackage{esint}
\usepackage{graphicx}
\usepackage{hyperref}
\usepackage{bm}
\usepackage{xcolor}

\newcommand*{\mailto}[1]{\href{mailto:#1}{\nolinkurl{#1}}}

\usepackage{todonotes}

\numberwithin{equation}{section}

\newtheorem{example}{Example}[section]

\newtheorem{theorem}[example]{Theorem}
 
\newtheorem{lemma}[example]{Lemma}
 \newtheorem{corollary}[example]{Corollary}

\newtheorem*{maintheorem*}{Main Theorem}
\allowdisplaybreaks
\numberwithin{equation}{section}

\renewcommand{\i}{\ifmmode\mathit{\mathchar"7010 }\else\char"10 \fi}
\renewcommand{\j}{\ifmmode\mathit{\mathchar"7011 }\else\char"11 \fi}

{%

\begin{enumerate}}%
{\end{enumerate}}

{%

\begin{enumerate}}%
{\end{enumerate}}

\begin{document}

\title[Embedding $H^s$ into $C^{0,s-\frac{1}{2}}$]
{H\"older continuity of functions in the fractional Sobolev spaces: 1-dimensional case}

\author[Rybalko]{Yan Rybalko}

\address[Yan Rybalko]{\newline 
	Mathematical Division, 
	\newline B.Verkin Institute for Low Temperature Physics and Engineering
	of the National Academy of Sciences of Ukraine,
	\newline 47 Nauky Ave., Kharkiv, 61103, Ukraine}
\email[]{rybalkoyan@gmail.com}




\date{\today}

\begin{abstract}
	This paper deals with the embedding of the Sobolev spaces of fractional order into the space of H\"older continuous functions.
	More precisely, we show that the function $f\in H^s(\mathbb{R})$ with 
	$\frac{1}{2}<s<1$ is H\"older continuous with the exponent $s-\frac{1}{2}$.
	This is a particular case of the much stronger embedding theorems (see Section 2.8.1 in \textit{H. Triebel, Interpolation Theory, Function Spaces, Differential Operators, North-Holland Pub. Co., Amsterdam, 1978.}), but here we give an elementary proof for $H^{s}(\mathbb{R})$.
\end{abstract}

\maketitle

\tableofcontents

\section{Introduction}
The Sobolev embedding theorem (see, e.g., \cite[equation (A.13)]{T06})
$$
\|f\|_{L^{\infty}(\mathbb{R}^n)}\lesssim_{p,s,n}
\|f\|_{W^{s,p}(\mathbb{R}^n)},
$$
which holds for $1<p<\infty$ and $s>0$ such that 
$\frac{1}{p}<\frac{s}{d}$, is all important in the theory of PDEs.
In particular, for $n=1$ and $p=2$ it reads as follows:
\begin{equation}\label{Sob-emb}
\|f\|_{L^{\infty}(\mathbb{R})}\lesssim_s
\|f\|_{H^s(\mathbb{R})},\quad s>\frac{1}{2}.
\end{equation}
The counterpart of the Sobolev inequality is the Morrey's inequality, which has the following form (see, e.g., \cite[Exercise A.22]{T06})
$$
\|f\|_{C^{0,\gamma}(\mathbb{R}^n)}\lesssim_{p,n}
\|f\|_{W^{1,p}(\mathbb{R}^n)},\quad \gamma=1-\frac{n}{p},
$$
where $n<p\leq\infty$, $\gamma=1-\frac{n}{p}$ and
\begin{equation}\label{Hold-norm}
\|f\|_{C^{0,\gamma}(\mathbb{R}^n)}=
\|f\|_{L^\infty(\mathbb{R}^n)}
+\sup\limits_{\xi_1\neq\xi_2}\left\{
\frac{|f(\xi_1)-f(\xi_2)|}{|\xi_1-\xi_2|^\gamma}
\right\},
\end{equation}
is the $\gamma$-th H\"older norm. In the case $n=1$ and $p=2$, the Morrey's inequality reduces to
\begin{equation}\label{Morrey}
\|f\|_{C^{0,1/2}(\mathbb{R})}\lesssim
\|f\|_{H^1(\mathbb{R})}.
\end{equation}

Combining \eqref{Sob-emb} and \eqref{Morrey} one concludes that any function $f\in H^s(\mathbb{R})$ with 
$\frac{1}{2}<s<1$ is continuous, while at the endpoint case, namely $s=1$, the function $f$ is H\"older continuous with the exponent $\frac{1}{2}$.
In the present work we establish that the function $f\in H^s(\mathbb{R})$ with 
$\frac{1}{2}<s<1$ is H\"older continuous with exponent $s-\frac{1}{2}$.
In particular, we have (see Corollary \ref{cor-1})
\begin{equation}\label{Hold-Sob}
\|f\|_{C^{0,s-1/2}(\mathbb{R})}\lesssim_s
\|f\|_{H^s(\mathbb{R})},\quad \frac{1}{2}<s<1,
\end{equation}
which is an improvement of the Sobolev embedding theorem \eqref{Sob-emb} in the one dimensional case.

The article is organized as follows.
In Section \ref{not} we introduce notations and some basic facts used in the paper.
In Section \ref{Emb} we first prove two auxiliary lemmas and then Theorem \ref{Hold-Sob-th}, where the main result is formulated.

\section{Notations}\label{not}
We write the Fourier transform of the function $f(\xi)$ as follows:
$$
\mathcal{F}(f)(k)=
\frac{1}{2\pi}\int_{\mathbb{R}}
e^{-\mathrm{i}k\xi}f(\xi)
\,d\xi,\quad k\in\mathbb{R},
$$
with the inverse relation given by
$$
\mathcal{F}^{-1}(F)(\xi)=
\int_{\mathbb{R}}
e^{\mathrm{i}k\xi}F(k)\,dk,\quad
\xi\in\mathbb{R}.
$$
In these notations the Plancherel formula has the form
\begin{equation}\label{Planch}
(f,g)_{L^2(\mathbb{R})}=2\pi(\mathcal{F}(f),\mathcal{F}(g))
_{L^2(\mathbb{R})}.
\end{equation}
The norm of the function $f(\xi)$ in the fractional Sobolev 
space $H^s(\mathbb{R})$, $s\in\mathbb{R}$ \cite{A55, G58, S58} is defined by
$$
\|f\|_{H^s(\mathbb{R})}^2=
2\pi\int_\mathbb{R}(1+|k|^2)^s|\mathcal{F}(f)(k)|^2\,dk.
$$
Finally, we denote the Schwartz space by $\mathcal{S}(\mathbb{R})$.

\section{Embedding $H^s$ into $C^{0,s-\frac{1}{2}}$}
\label{Emb}
At first, we prove two lemmas, which will be then applied in Theorem \ref{Hold-Sob-th}.
The first lemma is as follows:
\begin{lemma}\label{L-E-1-Lip}
	Consider the function
	\begin{equation}\label{E-1}
	E_1(\xi;\xi_1,\xi_2)=
	\begin{cases}
	e^{\xi-\xi_1}-e^{\xi-\xi_2},&\xi\leq\xi_1,\\
	0,&\mbox{otherwise},
	\end{cases}
	\end{equation}
	with the parameters
	$\xi_1,\xi_2\in\mathbb{R}$, $\xi_1\leq\xi_2$.
	Then the following estimate holds:
	\begin{equation}\label{E-1-Lip}
	\left(\int_\mathbb{R}|k|^{2\sigma}
	|\mathcal{F}(E_1((\cdot);\xi_1,\xi_2))(k)|^2
	\,dk\right)^{\frac{1}{2}}
	\lesssim_\sigma(\xi_2-\xi_1),
	\end{equation}
	for any $0<\sigma<\frac{1}{2}$.
\end{lemma}
\begin{proof}
	Throughout the proof we write $E_1(\xi)$ instead of $E_1(\xi;\xi_1,\xi_2)$ for simplicity.
	Taking into account that (see, e.g, \cite[equation (3.12)]{DNPV12})
	\begin{equation}\label{k2s}
	|k|^{2\sigma}=
	C(\sigma)\int_\mathbb{R}\frac{1-\cos kz}{|z|^{1+2\sigma}}\,dz,\
	\quad C(\sigma)=\left(\int_{\mathbb{R}}
	\frac{1-\cos\xi}{|\xi|^{1+2\sigma}}\,d\xi\right)^{-1},
	\end{equation}
	we obtain
	\begin{equation}
	\begin{split}
	\label{int-k2r}
	\int_\mathbb{R}|k|^{2\sigma}|\mathcal{F}(E_1)(k)|^2\,dk&=
	C(\sigma)\int_\mathbb{R}|\mathcal{F}(E_1)(k)|^2
	\int_\mathbb{R}\frac{1-\cos kz}{|z|^{1+2\sigma}}\,dz
	\,dk\\
	&=\frac{C(\sigma)}{2}\int_\mathbb{R}
	\frac{1}{|z|^{1+2\sigma}}
	\int_\mathbb{R}|e^{\mathrm{i}kz}-1|^2
	|\mathcal{F}(E_1)(k)|^2\,dk\,dz.
	\end{split}
	\end{equation}
	Then using that
	\begin{equation}
	\nonumber
	\int_\mathbb{R}|e^{\mathrm{i}kz}-1|^2
	|\mathcal{F}(E_1)(k)|^2\,dk=
	\|(e^{\mathrm{i}kz}-1)\mathcal{F}(E_1)(k)
	\|_{L^2}^2
	=\|\mathcal{F}(E_1(z+\cdot)-E_1(\cdot))(k)\|
	_{L^2}^2,
	\end{equation}
	and applying the Plancherel formula \eqref{Planch}, we have from \eqref{int-k2r}
	\begin{equation}
	\label{int-k2r-d}
	\begin{split}
	\int_\mathbb{R}|k|^{2\sigma}|
	\mathcal{F}(E_1)(k)|^2\,dk&=
	\frac{C(\sigma)}{4\pi}\int_\mathbb{R}\int_\mathbb{R}
	\frac{\left(E_1(z+\xi)-E_1(\xi)\right)^2}
	{|z|^{1+2\sigma}}\,d\xi\,dz\\
	&=\frac{C(\sigma)}{4\pi}
	\int_\mathbb{R}\int_\mathbb{R}
	\frac{\left(E_1(\xi)-E_1(\eta)\right)^2}
	{|\xi-\eta|^{1+2\sigma}}\,d\xi\,d\eta.
	\end{split}
	\end{equation}
	Recalling that $E_1(\xi)=0$ for $\xi>\xi_1$, the double integral in the right-hand side of \eqref{int-k2r-d} can be written in the form
	\begin{equation}
	\label{int-k2s-sum}
	\begin{split}
	\int_\mathbb{R}\int_\mathbb{R}
	\frac{\left(E_1(\xi)-E_1(\eta)\right)^2}
	{|\xi-\eta|^{1+2\sigma}}\,d\xi\,d\eta=&
	\int_{\xi_1}^{\infty}\int_{-\infty}^{\xi_1}
	\frac{E_1^2(\xi)}{|\xi-\eta|^{1+2\sigma}}
	\,d\xi\,d\eta
	+
	\int_{-\infty}^{\xi_1}\int_{\xi_1}^{\infty}
	\frac{E_1^2(\eta)}{|\xi-\eta|^{1+2\sigma}}
	\,d\xi\,d\eta\\
	&+
	\int_{-\infty}^{\xi_1}\int_{-\infty}^{\xi_1}
	\frac{\left(E_1(\xi)-E_1(\eta)\right)^2}
	{|\xi-\eta|^{1+2\sigma}}
	\,d\xi\,d\eta
	\equiv\sum\limits_{i=1}^3\tilde I_{i}.
	\end{split}
	\end{equation}
	For the integral $\tilde I_{2}$ we have
	\begin{equation}
	\nonumber
	\begin{split}
	\tilde I_{2}&=
	\left(e^{-\xi_1}-e^{-\xi_2}\right)^2
	\int_{-\infty}^{\xi_1}e^{2\eta}
	\int_{\xi_1}^{\infty}
	(\xi-\eta)^{-1-2\sigma}
	\,d\xi\,d\eta\\
	&\leq \frac{e^{-2\xi_1}}{2\sigma}
	(\xi_2-\xi_1)^2\int_{-\infty}^{\xi_1}e^{2\eta}
	(\xi_1-\eta)^{-2\sigma}\,d\eta
	\lesssim_\sigma(\xi_2-\xi_1)^2
	\int_{-\infty}^0e^{2\eta}|\eta|^{-2\sigma}
	\,d\eta
	\lesssim_\sigma (\xi_2-\xi_1)^2,
	\end{split}
	\end{equation}
	where in the last inequality we use that $0<\sigma<\frac{1}{2}$.
	The integral $\tilde I_{1}$, after changing the order of integration, can be estimated similarly.
	
	Now consider $\tilde I_{3}$.
	We write this integral in the following form:
	\begin{equation}
	\nonumber
	\left(
	\int_{-\infty}^{\xi_1}\int_{-\infty}^{\eta-1}
	+\int_{-\infty}^{\xi_1}\int_{\eta-1}^{\eta+1}
	+\int_{-\infty}^{\xi_1}\int_{\eta+1}^{\xi_1}
	\right)
	\frac{\left(E_1(\xi)-E_1(\eta)\right)^2}
	{|\xi-\eta|^{1+2\sigma}}
	\,d\xi\,d\eta
	\equiv\sum\limits_{i=1}^3\tilde{I}_{3,i}.
	\end{equation}
	Taking into account that 
	$E_1(\xi)-E_1(\eta)=
	(e^{-\xi_1}-e^{-\xi_2})(e^\xi-e^\eta)$ and $|e^a-e^b|\leq
	e^{\max\{a,b\}}|a-b|$, we have for $\tilde I_{3,1}$
	\begin{equation}\label{tilde-I-3-1}
	\begin{split}
	\tilde I_{3,1}&=\left(
	e^{-\xi_1}-e^{-\xi_2}\right)^2
	\int_{-\infty}^{\xi_1}\int_{-\infty}^{\eta-1}
	\frac{(e^{\xi}-e^{\eta})^2}
	{(\eta-\xi)^{1+2\sigma}}\,d\xi\,d\eta\\
	&\leq e^{-2\xi_1}(\xi_2-\xi_1)^2
	\left(
	\int_{-\infty}^{\xi_1}\int_{-\infty}^{\eta-1}
	\frac{e^{2\eta}}
	{(\eta-\xi)^{1+2\sigma}}\,d\xi\,d\eta
	+
	\int_{-\infty}^{\xi_1}\int_{-\infty}^{\eta-1}
	\frac{e^{2\xi}}
	{(\eta-\xi)^{1+2\sigma}}\,d\xi\,d\eta
	\right).
	\end{split}
	\end{equation}
	The integrals in the right-hand side of \eqref{tilde-I-3-1} can be calculated as follows
	$$
	\int_{-\infty}^{\xi_1}\int_{-\infty}^{\eta-1}
	\frac{e^{2\eta}}
	{(\eta-\xi)^{1+2\sigma}}\,d\xi\,d\eta
	=\frac{1}{2\sigma}\int_{-\infty}^{\xi_1}
	e^{2\eta}\,d\eta=\frac{1}{4\sigma}e^{2\xi_1},
	$$
	and 
	\begin{equation}
	\nonumber
	\begin{split}
	\int_{-\infty}^{\xi_1}\int_{-\infty}^{\eta-1}
	\frac{e^{2\xi}}
	{(\eta-\xi)^{1+2\sigma}}\,d\xi\,d\eta
	&=\int_{-\infty}^{\xi_1-1}e^{2\xi}
	\int_{\xi+1}^{\xi_1}
	(\eta-\xi)^{-1-2\sigma}
	\,d\eta\,d\xi\\
	&=
	\frac{e^{2\xi_1}}{2\sigma}
	\int_{-\infty}^{\xi_1-1}e^{2(\xi-\xi_1)}
	\left(1-(\xi_1-\xi)^{-2\sigma}\right)\,d\xi
	\lesssim_\sigma e^{2\xi_1},
	\end{split}
	\end{equation}
	which imply that $\tilde I_{3,1}\lesssim_\sigma (\xi_2-\xi_1)^2$.
	To deal with the integral $\tilde I_{3,3}$ one argue similarly as for $\tilde{I}_{3,1}$ and use, in addition, that
	$0<\sigma<\frac{1}{2}$ and
	$$
	\int_{-\infty}^{\xi_1}\int_{\eta+1}^{\xi_1}
	\frac{e^{2\xi}}{(\xi-\eta)^{1+2\sigma}}\,d\xi\,d\eta
	\leq
	\int_{-\infty}^{\xi_1-1}\int_{\eta+1}^{\xi_1}
	\frac{e^{2\xi}}{(\xi-\eta)^{1+2\sigma}}\,d\xi\,d\eta
	=\int_{-\infty}^{\xi_1}\int_{-\infty}^{\xi-1}
	\frac{e^{2\xi}}{(\xi-\eta)^{1+2\sigma}}
	\,d\eta\,d\xi.
	$$

	Finally, let us estimate the integral 
	$\tilde I_{3,2}$.
	Taking into account that $|e^\xi-e^\eta|\leq
	e^{\max\{\xi,\eta\}}|\xi-\eta|$, we have
	\begin{equation}
	\nonumber
	\tilde I_{3,2}
	\leq e^{-2\xi_1}(\xi_2-\xi_1)^2
	\left(
	\int_{-\infty}^{\xi_1}\int_{\eta-1}^{\eta}
	e^{2\eta}
	(\eta-\xi)^{1-2\sigma}\,d\xi\,d\eta
	+
	\int_{-\infty}^{\xi_1}\int_{\eta}^{\eta+1}
	e^{2\xi}
	(\xi-\eta)^{1-2\sigma}\,d\xi\,d\eta
	\right).
	\end{equation}
	Using that $2-2\sigma>0$, the above integrals can be estimated as follows:
	$$
	\int_{-\infty}^{\xi_1}\int_{\eta-1}^{\eta}
	e^{2\eta}
	(\eta-\xi)^{1-2\sigma}\,d\xi\,d\eta
	=\frac{1}{2-2\sigma}\int_{-\infty}^{\xi_1}
	e^{2\eta}\,d\eta\lesssim_\sigma e^{2\xi_1},
	$$
	and
	$$
	\int_{-\infty}^{\xi_1}\int_{\eta}^{\eta+1}
	e^{2\xi}
	(\xi-\eta)^{1-2\sigma}\,d\xi\,d\eta
	\leq
	\int_{-\infty}^{\xi_1+1}\int_{\xi-1}^{\xi}
	e^{2\xi}
	(\xi-\eta)^{1-2\sigma}\,d\eta\,d\xi
	\lesssim_\sigma e^{2\xi_1},
	$$
	Thus eventually we have that $\tilde I_{i}\lesssim_\sigma(\xi_2-\xi_1)^2$, $i=1,2,3$, which, together with \eqref{int-k2r-d} and \eqref{int-k2s-sum}, imply \eqref{E-1-Lip}.
\end{proof}

The second lemma reads
\begin{lemma}\label{L-E-2-Hold}
	Consider the function
	\begin{equation}\label{E-2}
	E_2(\xi;\xi_1,\xi_2)=
	\begin{cases}
	e^{\xi-\xi_2},&
	\xi_1\leq\xi\leq\xi_2,\\
	0,&\mbox{otherwise},
	\end{cases},
	\end{equation}
	with the parameters
	$\xi_1,\xi_2\in\mathbb{R}$ such that $0\leq\xi_2-\xi_2\leq1$.
	Then the following estimate holds:
	\begin{equation}\label{E-2-Hold}
	\left(\int_\mathbb{R}|k|^{2\sigma}
	|\mathcal{F}(E_2((\cdot);\xi_1,\xi_2))(k)|^2
	\,dk\right)^{\frac{1}{2}}
	\lesssim_\sigma(\xi_2-\xi_1)
	^{\frac{1}{2}-\sigma},
	\end{equation}
	for any $0<\sigma<\frac{1}{2}$.
\end{lemma}
\begin{proof}
	Arguing similarly as in Lemma \ref{L-E-1-Lip} (see \eqref{k2s}--\eqref{int-k2s-sum}), we
	obtain (here we drop the dependence of $E_2$ on the parameters $\xi_1$ and $\xi_2$)
	\begin{equation}\label{E-2-2s}
	\begin{split}
	\int_\mathbb{R}|k|^{2\sigma}
	|\mathcal{F}(E_2)(k)|^2
	\,dk\lesssim_\sigma&
	\int_{-\infty}^{\xi_1}\int_{\xi_1}^{\xi_2}
	\frac{E_2^2(\xi)}{|\xi-\eta|^{1+2\sigma}}
	\,d\xi\,d\eta
	+
	\int_{\xi_2}^{\infty}\int_{\xi_1}^{\xi_2}
	\frac{E_2^2(\xi)}{|\xi-\eta|^{1+2\sigma}}
	\,d\xi\,d\eta\\
	&+\int_{\xi_1}^{\xi_2}\int_{-\infty}^{\xi_1}
	\frac{E_2^2(\eta)}{|\xi-\eta|^{1+2\sigma}}
	\,d\xi\,d\eta
	+
	\int_{\xi_1}^{\xi_2}\int_{\xi_2}^{\infty}
	\frac{E_2^2(\eta)}{|\xi-\eta|^{1+2\sigma}}
	\,d\xi\,d\eta\\
	&+
	\int_{\xi_1}^{\xi_2}\int_{\xi_1}^{\xi_2}
	\frac{\left(E_2(\xi)-E_2(\eta)\right)^2}
	{|\xi-\eta|^{1+2\sigma}}
	\,d\xi\,d\eta
	\equiv\sum\limits_{i=1}^5\check I_{i}.
	\end{split}
	\end{equation}
	The integral $\check I_{3}$ can be estimated as follows (recall that ($1-2\sigma>0$)):
	\begin{equation}
	\nonumber
	\begin{split}
	\check I_{3}&=\frac{e^{-2\xi_2}}{2\sigma}
	\int_{\xi_1}^{\xi_2}e^{2\eta}
	(\eta-\xi_1)^{-2\sigma}\,d\eta=
	\frac{(\xi_2-\xi_1)^{1-2\sigma}}
	{2\sigma(1-2\sigma)}
	-\frac{e^{-2\xi_2}}{\sigma(1-2\sigma)}
	\int_{\xi_1}^{\xi_2}e^{2\eta}
	(\eta-\xi_1)^{1-2\sigma}\,d\eta\\
	&\leq \frac{(\xi_2-\xi_1)^{1-2\sigma}}
	{2\sigma(1-2\sigma)}.
	\end{split}
	\end{equation}
	The integral $\check I_1$ can be treated in the same way after changing the order of integration.
	For $\check I_4$ we have ($\check I_{2}$ can be analyzed similarly after, again, changing the order of integration)
	\begin{equation}
	\nonumber
	\begin{split}
	\check I_{4}&=\frac{e^{-2\xi_2}}{2\sigma}
	\int_{\xi_1}^{\xi_2}e^{2\eta}
	(\xi_2-\eta)^{-2\sigma}\,d\eta=
	\frac{e^{2(\xi_1-\xi_2)}}
	{2\sigma(1-2\sigma)}
	(\xi_2-\xi_1)^{1-2\sigma}\\
	&\quad+\frac{e^{-2\xi_2}}{\sigma(1-2\sigma)}
	\int_{\xi_1}^{\xi_2}e^{2\eta}
	(\xi_2-\eta)^{1-2\sigma}\,d\eta
	\leq \frac{(\xi_2-\xi_1)^{1-2\sigma}}
	{2\sigma(1-2\sigma)}
	+\frac{(\xi_2-\xi_1)^{2-2\sigma}}
	{\sigma(1-2\sigma)}\\
	&\lesssim_\sigma(\xi_2-\xi_1)^{1-2\sigma},
	\end{split}
	\end{equation}
	where we use that $\xi_2-\xi_1\leq1$ and
	$\int_{\xi_1}^{\xi_2}f(\eta)g(\eta)\,d\eta
	\leq(\xi_2-\xi_1)
	\max\limits_{\xi_1\leq\eta\leq\xi_2}f(\eta)\cdot
	\max\limits_{\xi_1\leq\eta\leq\xi_2}g(\eta)
	$.
	
	Finally, to estimate $\check I_5$ we, taking into account that $|e^a-e^b|\leq
	e^{\max\{a,b\}}|a-b|$, observe that
	\begin{equation}\label{ch-I-5}
	\begin{split}
	\check I_5&=e^{-2\xi_2}
	\int_{\xi_1}^{\xi_2}\int_{\xi_1}^{\eta}
	\frac{\left(e^\xi-e^\eta\right)^2}
	{(\eta-\xi)^{1+2\sigma}}\,d\xi\,d\eta
	+e^{-2\xi_2}
	\int_{\xi_1}^{\xi_2}\int_{\eta}^{\xi_2}
	\frac{\left(e^\xi-e^\eta\right)^2}
	{(\xi-\eta)^{1+2\sigma}}\,d\xi\,d\eta\\
	&\leq e^{-2\xi_2}
	\int_{\xi_1}^{\xi_2}e^{2\eta}
	\int_{\xi_1}^{\eta}
	(\eta-\xi)^{1-2\sigma}\,d\xi\,d\eta
	+e^{-2\xi_2}
	\int_{\xi_1}^{\xi_2}
	\int_{\eta}^{\xi_2}e^{2\xi}
	(\xi-\eta)^{1-2\sigma}\,d\xi\,d\eta.
	\end{split}
	\end{equation}
	The first integral in the right-hand side of \eqref{ch-I-5} has the form (the second integral can be estimated similarly after changing the order of integration)
	\begin{equation}
	\nonumber
	\int_{\xi_1}^{\xi_2}e^{2\eta}
	\int_{\xi_1}^{\eta}
	(\eta-\xi)^{1-2\sigma}\,d\xi\,d\eta
	=\frac{1}{2-2\sigma}	
	\int_{\xi_1}^{\xi_2}e^{2\eta}
	(\eta-\xi_1)^{2-2\sigma}\,d\eta
	\lesssim_\sigma e^{2\xi_2}
	(\xi_2-\xi_1)^{1-2\sigma}.
	\end{equation}
	Thus we eventually have that 
	$\check I_i\lesssim_\sigma 
	(\xi_2-\xi_1)^{1-2\sigma}$ for all $i=1,\dots,5$, which, together with \eqref{E-2-2s}, imply 
	\eqref{E-2-Hold}
\end{proof}

Now we are at the position to prove the main result of this paper, which states the H\"older continuity of functions $f\in H^{s}(\mathbb{R})$, $\frac{1}{2}<s<1$.
\begin{theorem}\cite[Section 2.8.1]{T78}\label{Hold-Sob-th}
	Suppose that $f(\xi)\in H^{s}(\mathbb{R})$ with $\frac{1}{2}<s<1$.
	Then $f(\xi)$ is locally H\"older continuous with the exponent $s-\frac{1}{2}$:
	\begin{equation}\label{Hr-Hold}
	|f(\xi_1)-f(\xi_2)|\lesssim_s
	\|f\|_{H^s(\mathbb{R})}
	|\xi_1-\xi_2|^{s-\frac{1}{2}},
	\end{equation}
	for all $\xi_1,\xi_2\in\mathbb{R}$ such that
	$|\xi_1-\xi_2|\leq 1$.
\end{theorem}
\begin{proof}
	Without loss of generality we can take 
	$0\leq\xi_2-\xi_1\leq1$.
	Consider a function $f\in\mathcal{S}(\mathbb{R})$.
	Then for any $\xi_1\in\mathbb{R}$ we have
	\begin{equation}\label{f-repr}
	f(\xi_1)=\int_{-\infty}^{\xi_1}
	\frac{d}{d\xi}\left[
	e^{\xi-\xi_1}f(\xi)
	\right]\,d\xi
	=\int_{-\infty}^{\xi_1}
	e^{\xi-\xi_1}f(\xi)\,d\xi
	+\int_{-\infty}^{\xi_1}
	e^{\xi-\xi_1}f^\prime(\xi)\,d\xi.
	\end{equation}
	Let 
	\begin{equation}
	\nonumber
	I=	\int_{-\infty}^{\xi_1}e^{\xi-\xi_1}
	f^\prime(\xi)\,d\xi
	-\int_{-\infty}^{\xi_2}e^{\xi-\xi_2}
	f^\prime(\xi)\,d\xi.
	\end{equation}
	Then using \eqref{f-repr}, the Sobolev embedding theorem \eqref{Sob-emb} and that $|e^a-e^b|\leq
	e^{\max\{a,b\}}|a-b|$,
	we obtain 
	\begin{equation}\label{f-diff}
	\begin{split}
	\left|f(\xi_1)-f(\xi_2)\right|
	&\leq
	\left|
	\int_{-\infty}^{\xi_1}
	e^\xi(e^{-\xi_1}-e^{-\xi_2})f(\xi)\,d\xi
	\right|
	+\left|
	\int_{\xi_1}^{\xi_2}e^{\xi-\xi_2}f(\xi)\,d\xi
	\right|
	+|I|\\
	&\lesssim_s \|f\|_{H^s}e^{-\xi_1}(\xi_2-\xi_1)
	\int_{-\infty}^{\xi_1} e^\xi\,d\xi
	+\|f\|_{H^s}e^{-\xi_2}
	\int_{\xi_1}^{\xi_2}e^\xi\,d\xi
	+|I|
	\\
	&\lesssim_s\|f\|_{H^s}(\xi_2-\xi_1)+|I|,
	\end{split}
	\end{equation}
	To prove that $|I|$ is H\"older continuous with the exponent $s-\frac{1}{2}$, we write $I$ in the form
	\begin{equation}\label{I-1-I-2}
	I=\int_{-\infty}^{\infty}
	E_1(\xi)f^\prime(\xi)\,d\xi
	-\int_{-\infty}^{\infty}E_2(\xi)
	f^\prime(\xi)\,d\xi\equiv I_1+I_2,
	\end{equation}
	with $E_1(\xi)=E_1(\xi;\xi_1,\xi_2)$ and
	$E_2(\xi)=E_2(\xi;\xi_1,\xi_2)$ given by \eqref{E-1} and \eqref{E-2} respectively.
	
	Let us estimate the integral $I_1$.
	By the Cauchy-Schwarz inequality and the Plancherel identity \eqref{Planch} we have
	\begin{equation}\label{|I_1|}
	\begin{split}
	|I_1|&\leq2\pi\|f^\prime\|_{H^{s-1}}
	\left(\int_{-\infty}^{\infty}
	(1+k^2)^{1-s}|\mathcal{F}(E_1)(k)|^2\,dk
	\right)^{\frac{1}{2}}\\
	&\lesssim_s\|f\|_{H^s}\|E_1\|_{L^2}
	+\|f\|_{H^s}\left(\int_{-\infty}^{\infty}
	|k|^{2-2s}|\mathcal{F}(E_1)(k)|^2\,dk
	\right)^{\frac{1}{2}}.
	\end{split}
	\end{equation}
	The $L^2$ norm of $E_1$ clearly satisfies the Lipschitz property:
	\begin{equation}
	\nonumber
	\|E_1\|_{L^2}=
	\left(e^{-\xi_1}-e^{-\xi_2}\right)
	\left(
	\int_{-\infty}^{\xi_1}e^{2\xi}\,d\xi
	\right)^{\frac{1}{2}}
	\leq\frac{1}{\sqrt{2}}(\xi_2-\xi_1),
	\end{equation}
	which, together with Lemma \ref{L-E-1-Lip} with $\sigma=1-s$,
	imply that
	\begin{equation}\label{I-1-Hold} |I_1|\lesssim_s\|f\|_{H^s}(\xi_2-\xi_1).
	\end{equation}
	
	Arguing similarly as in \eqref{|I_1|} and using Lemma \ref{L-E-2-Hold} with $\sigma=1-s$, we obtain the following estimate for $|I_2|$:
	\begin{equation}\label{I-2-Hold}
	\begin{split}
	|I_2|&\lesssim_s\|f\|_{H^s}
	\left(\int_{\xi_1}^{\xi_2}e^{2\xi-2\xi_2}
	\,d\xi\right)^{\frac{1}{2}}
	+\|f\|_{H^s}
	\left(\int_{-\infty}^{\infty}
	|k|^{2-2s}|\mathcal{F}(E_2)(k)|^2\,dk
	\right)^{\frac{1}{2}}\\
	&\lesssim_s\|f\|_{H^s}
	\left((\xi_2-\xi_1)^{\frac{1}{2}}
	+(\xi_2-\xi_1)^{s-\frac{1}{2}}
	\right)
	\lesssim_s\|f\|_{H^s}
	(\xi_2-\xi_1)^{s-\frac{1}{2}}.
	\end{split}
	\end{equation}
	Combining \eqref{f-diff}, \eqref{I-1-I-2}, \eqref{I-1-Hold} and \eqref{I-2-Hold}, we obtain \eqref{Hr-Hold} for $f\in\mathcal{S}(\mathbb{R})$.
	
	To prove \eqref{Hr-Hold} for $f\in H^s(\mathbb{R})$, we approximate $f$ by a sequence of Schwartz functions $f_n\in\mathcal{S}$, $n\in\mathbb{N}$ (here we use the fact that the Schwartz space $\mathcal{S}(\mathbb{R})$ is dense in $H^s(\mathbb{R})$, see \cite[Theorem 7.38]{A75}).
	The Sobolev embedding \ref{Sob-emb} implies that $f_n(\xi)\to f(\xi)$ for all $\xi\in\mathbb{R}$ and we can take a limit $n\to\infty$ in the inequality
	$$
	|f_n(\xi_1)-f_n(\xi_2)|
	\lesssim_s\|f_n\|_{H^s(\mathbb{R})}
	|\xi_1-\xi_2|^{s-\frac{1}{2}},
	$$
	for all $\xi_1,\xi_2\in\mathbb{R}$ such that
	$|\xi_1-\xi_2|\leq 1$.
\end{proof}
\begin{corollary}\label{cor-1}
	Theorem \ref{Hold-Sob-th} implies inequality \eqref{Hold-Sob}.
	\begin{proof}
		Indeed, from the definition of the H\"older norm (see \eqref{Hold-norm}) the Sobolev embedding theorem \eqref{Sob-emb} and \eqref{Hr-Hold}, we have for any 
		$\frac{1}{2}<s<1$
		\begin{equation}
		\nonumber
		\begin{split}
		\|f\|_{C^{0,s-1/2}}&=
		\|f\|_{L^\infty}
		+\sup\limits_{0<|\xi_1-\xi_2|\leq1}\left\{
		\frac{|f(\xi_1)-f(\xi_2)|}
		{|\xi_1-\xi_2|^{s-\frac{1}{2}}}
		\right\}+
		\sup\limits_{|\xi_1-\xi_2|>1}\left\{
		\frac{|f(\xi_1)-f(\xi_2)|}
		{|\xi_1-\xi_2|^{s-\frac{1}{2}}}
		\right\}\\
		& \lesssim_s\|f\|_{H^s}.
		\end{split}
		\end{equation}
	\end{proof}
\end{corollary}

\end{document}